\documentclass{amsart}
\usepackage{graphicx} 
\usepackage{url}
\usepackage[style=alphabetic]{biblatex}
\usepackage{multirow}
\usepackage{array}
\renewcommand{\arraystretch}{1}
\usepackage{booktabs}
\usepackage{amsfonts}
\usepackage{amsthm}
\usepackage{amsmath}
\usepackage{outlines}
\usepackage{amssymb}
\usepackage[top=1in,bottom=1in,left=1in,right=1in,twoside=false]{geometry}
\usepackage[colorlinks=true,linkcolor=blue]{hyperref}
\usepackage{algpseudocode}
\usepackage{indentfirst}
\usepackage{float}
\usepackage{placeins}
\usepackage{longtable}
\usepackage{booktabs}
\usepackage{tabularx}

\newcolumntype{C}[1]{>{\centering\arraybackslash}p{#1}}

\newtheorem{theorem}{Theorem}[section]
\newtheorem{lemma}[theorem]{Lemma}
\newtheorem{corollary}[theorem]{Corollary}

\newtheorem{proposition}[theorem]{Proposition}
\newtheorem{conjecture}[theorem]{Conjecture}
\theoremstyle{definition}
\newtheorem{definition}{Definition}[section]
\newtheorem{remark}[theorem]{Remark}
\newtheorem{question}{Question}

\title[Genus 3 Hyperelliptic Jacobians over Finite Fields]{Hyperelliptic Jacobians in Isogeny Classes of Abelian Threefolds Over Finite Fields}

\author{Matvey~Borodin}
\address{Department of Mathematics, Massachusetts Institute of Technology, 182 Memorial Dr, Cambridge, MA 02139}
\email{matveyborodin1@gmail.com}

\author{Liam~May}
\address{Department of Mathematics, Massachusetts Institute of Technology, 182 Memorial Dr, Cambridge, MA 02139}
\email{liammay2@mit.edu}

\thanks{The first author was supported by the UROP Fund for First-Year and Sophomore students.}
\thanks{The second author was supported by the Paul E. Gray (1954) UROP Fund.}

\date{August 2025}

\newcommand{\F}{\mathbb{F}}
\newcommand{\Z}{\mathbb{Z}}

\bibliography{bibliography}

\begin{document}

\begin{abstract}
    We present new criteria that obstruct an isogeny class of abelian varieties over a finite field with a given Weil polynomial from containing a Jacobian of a genus-3 hyperelliptic curve. Based on our analysis of the Weil polynomials of three-dimensional abelian varieties over finite fields up to $\F_{25}$ using the data in the L-functions and Modular Forms Database \cite{lmfdb}, we conjecture a collection of apparent obstructions. We provide a survey of known and conjectured results related to this problem, and a detailed statistical analysis of these findings. We conjecture that two of these obstructions classify all isogeny classes asymptotically as $q \to \infty$.

\end{abstract}

\maketitle

\section{Introduction}

The Honda-Tate theorem \cite{HONDA, tate} associates to each isogeny class $\mathcal{A}$ of $g$-dimensional abelian varieties over $\F_q$ a unique degree-$2g$ Weil polynomial with roots of modulus $\sqrt{q}$.  The Weil polynomial of an abelian variety $A$ is the characteristic polynomial of its Frobenius endomorphism and is unique up to isogeny. The Jacobians of genus $g$ hyperelliptic curves over $\F_q$ form a subset of abelian varieties of particular interest, as hyperelliptic curves have proved, over the past decades, to be a rich area of research, both in connection with open problems in mathematics and cryptography and as standalone objects of interest. A natural question that arises is the following.

\begin{question} \label{question1}
    Given a Weil polynomial $f(x) = x^6 + sx^5 + tx^4 + ux^3 + tqx^2 + sq^2x + q^3$ of an isogeny class $\mathcal{A}$ of 3-dimensional abelian varieties over $\F_q$, does $\mathcal{A}$ contain the Jacobian of a genus 3 hyperelliptic curve? 
\end{question}

The analogous questions for genus 1 and 2 have been answered by \citeauthor{WATERHOUSE1969} \cite{WATERHOUSE1969} and \citeauthor{Howe2009} \cite{Howe2009}, respectively, in the form of short lists of easily computable arithmetic conditions. Answering Question~\ref{question1} for genus 3 curves results in substantially more complications, as has been noted by several authors, though some progress has been made. Question~\ref{question1} was answered for Jacobians of (not-necessarily hyperelliptic) curves in isogeny classes of supersingular abelian threefolds in characteristic 2 in \cite{nart2006}; however, the hyperelliptic case is quite different. Other results include the powerful ``Resultant-1'' method \cite{howe2022deducinginformationcurvesfinite}  discussed in Section~\ref{subsec:resultant_1}, an extremely common obstruction on coefficients modulo 2 \cite{costa2020}, and an obstruction on supersingular hyperelliptic curves over finite fields of characteristic 2 \cite{Scholten2001}. Additionally, if $f(x)$ factors as the product $f_E^3$, where $f_E$ is the Weil polynomial of an elliptic curve $E/\F_q$, there are a number of known obstructions in terms of the discriminant of $E$ found in \cite{SERRE1985, ZAYTSEV2014, ZAYTSEV2016}. There are also known obstructions stemming from \cite{howe2007improvedupperboundsnumber}. Algorithmic implementations of the last few of these, as well as some discussion regarding the results, can be found in \cite{davidIsogeny}. We also provide a detailed discussion of known obstructions in Section~\ref{sec:summary}. 

By thorough analysis of the data available on the L-functions and Modular Forms Database \cite{lmfdb}, we have collected a list of arithmetic obstructions to a Weil polynomial arising from a genus 3 hyperelliptic curve. In Section~\ref{sec:methodology}, we discuss our methodology for data analysis. All of our results are checked against the abelian variety data found on the \cite{lmfdb} (see \cite{dupuy2020isogenyclassesabelianvarieties} for more details about the LMFDB), which includes threefolds over finite fields up to $\F_{25}$.

In Section~\ref{sec:StatAndAsym}, we discuss the asymptotic relevance of the known and conjectured results. In particular, we claim that a small subset of the rules asymptotically classifies $100 \%$ of degree-6 Weil polynomials as $q \rightarrow \infty$.

Our findings are best viewed by separating results for fields of even and odd characteristic. In this paper, we give a broad overview of all our results, with the intent of presenting them as conjectures, discussing their effectiveness, and highlighting the remaining work to be done towards a complete classification. We outline our methodology used for analyzing the abelian variety data found on the LMFDB, which generalizes easily to arbitrary genus.  Additionally, we provide a census of all currently known results applicable to Question ~\ref{question1}, allowing this paper to serve as a complete reference for current progress on Question~\ref{question1}. In~\cite{borodinmayeven}, we give a proof of the main obstruction in characteristic 2 and provide an algorithm for enumerating hyperelliptic curves over characteristic 2 fields, a case that is omitted in \cite{howe2024enumeratinghyperellipticcurvesfinite}.

Our code is available at \url{https://github.com/bmatvey/genus_3_hyp_jacobians}.

\subsection*{Acknowledgments}

First and foremost, we would like to thank our wonderful mentor, David Roe, for suggesting this problem and supporting us throughout our endeavors with invaluable advice, both on the mathematics and our write-up. We would also like to thank Edgar Costa and Stefano Marseglia for taking the time to discuss our work with us and offer advice. We thank Everett Howe for reading a preliminary draft and offering invaluable remarks and observations regarding our results. Finally, we thank Professor Kiran Kedlaya for his help tracking down algorithms for characteristic 2 finite fields. The first author was supported by the UROP Fund for First-Year and Sophomore Students and the second author was supported by the Paul E. Gray (1954) UROP Fund while working on this project. 

\section{Methodology}\label{sec:methodology}

For each isogeny class of 3-dimensional abelian varieties over the finite fields $\F_q$ for $q \leq 25$, we used the following data from the \cite{lmfdb}:
\begin{enumerate}
    \item The associated $q$-Weil polynomial,
    \item The Newton polygon (and therefore the $p$-rank),
    \item The number of Weil polynomial factors,
    \item Whether the isogeny class contains a hyperelliptic Jacobian.
\end{enumerate}
We sorted the Weil polynomial data into categories based on these quantifiers. We performed extensive analysis of this data, with the intent of identifying conditions on the Weil coefficients in a given category that obstruct such isogeny classes from containing a hyperelliptic Jacobian. Candidate criteria were selected in the following manner: modular conditions, algebraic conditions, generalizations of results for genus 2, plotting, and inspection. 
\begin{enumerate}
    \item \textbf{Modular conditions:} Following the results in \cite{costa2020}, we automated the checking of all possible modular conditions of a similar type. In particular, for all $n \in \Z$ up to 10, we checked all conditions of the form $(s, t, u) \equiv (r_1, r_2, r_3) \pmod{n}$ for $(r_1, r_2, r_3) \in \Z/n\Z$. This yielded useful obstructions for $n = 2$ and 4, and would be worth investigating for higher genera. 

    \item \textbf{Algebraic conditions:} We automated checking for various algebraic conditions on the parameters, such as $p_1 s^{p_2} + p_3t^{p_4} + p_5 u^{p_6} = p_7$ and $p_1s + p_2t + p_3u + p_4s^2 + p_5t^2 + p_6u^2 = p_7$ for iterated $p_1, \ldots, p_7 \in \Z$ and $s, t, u \in \Z$ being the parameters classifying a Weil polynomial. This approach was limited by the speed of the computations (having more than seven parameters or large ranges for parameters was not realistic). The only interesting rule we were able to deduce from this approach was \textbf{1.3.N.0} (see Section~\ref{sec:summary}).
    
    \item \textbf{Generalization of genus 2 classification:} We used the results of \cite[Thm. 1.2]{Howe2009} as a blueprint for the form that a classification answering Question~\ref{question1} might take. Naturally, for each of the obstructions proved in \cite{Howe2009}, we automated searches to determine whether analogous rules applied to the case of genus 3. We discovered that only the rule concerning discriminants of products of powers of elliptic curves generalized easily to genus 3. The remaining rules either did not appear to generalize, or were subcases of the ``Resultant-1" obstruction, which we discuss in Section~\ref{subsec:resultant_1}.

    \item \textbf{Plotting:} Plotting the datasets in three dimensions allowed for easy identification of arithmetic patterns. We plotted data points corresponding to the three parameters of Weil polynomials, coloring data points based on whether the associated isogeny class contained a hyperelliptic Jacobian. The data was split by field, $p$-rank and factoring type. Generalizing this method to higher genera would be difficult due to dimensionality. 

    \item \textbf{Inspection:} Some rules became apparent by simple inspection of the numerical data. As the number of unclassified Weil polynomials in a given category approached zero, restricting our inspection to the unclassified data often allowed us to identify obstructions that had not been previously obvious.

\end{enumerate}

A given candidate rule was accepted as conjecture if for every Weil polynomial in the associated category satisfying the rule, the associated isogeny class did not contain a hyperelliptic Jacobian, and the number of such isogeny classes was sufficiently large. To be deemed ``sufficiently large'', a rule needed to have examples in the majority of fields to which it was applicable. An unavoidable flaw of this approach is that obstructions which apply to only a negligible number of isogeny classes were disregarded so as to avoid overfitting. As seen by analysis of the dimension 2 data on the LMFDB, some of the obstructions given in \cite{Howe2009} apply to a very small number of isogeny classes, and the same is likely true for dimension 3. 

It is possible that some of these obstructions are instances of overfitting and will not prove true for larger fields, but we expect that the majority will hold generally, since we made efforts to only accept obstructions that applied widely. We provide detailed statistics on the prevalence of all of the rules in Appendix~\ref{sec:stats}. The code used to generate the rules and statistics, and to automate the search for rules (eg. the plotting code), is available at \url{https://github.com/bmatvey/genus_3_hyp_jacobians}.

We also remark that based on our analysis, a complete classification analogous to the one for genus 2 in \cite{Howe2009} is likely unfeasible. As the proportion of unclassified data approached zero, it became increasingly hard to find further elementary obstructions. It appears that a complete classification in genus 3 would need to be significantly more complex than what was done in genus 2, rendering it impractical and unrealistic to write down such a classification. In Section~\ref{sec:StatAndAsym}, we argue that if we focus on the limiting case $q \rightarrow \infty$ we can mitigate this challenge by proving that a small subset of obstructions asymptotically classifies $100\%$ of Weil polynomials.

\section{Obstructions}\label{sec:summary}

In this section, we enumerate known and conjectured obstructions to an isogeny class of abelian varieties containing a Jacobian of a genus-3 hyperelliptic curve. In particular, an isogeny class of abelian varieties whose Weil polynomial satisfies any of the obstructions described below cannot contain a genus 3 hyperelliptic Jacobian. As detailed in Section~\ref{sec:methodology}, we have computationally verified that these hold for all Weil polynomials over finite fields up to $\F_{25}$.

For a Weil polynomial of an isogeny class of abelian varieties, $f(x) = x^6 + sx^5 + tx^4 + ux^3 + qtx^2 + q^2sx + q^3$, we use $s, t, u$ to refer to the coefficients characterizing the polynomial. It is often computationally simpler to consider the real Weil polynomial, which is the unique monic polynomial $g(x)$ such that $f(x) = x^3 g(x + q/x)$. If $g(x)$ factors into three linear terms (so $f(x)$ factors into three quadratics), we write $g(x) = (x + \alpha)(x + \beta)(x + \gamma)$ with $\alpha \leq \beta \leq \gamma$. When $g(x)$ factors into a linear term and a quadratic (so $f(x)$ factors into a quadratic and a quartic), we write $g(x) = (x+\alpha)(x^2 + \delta x + \epsilon)$. Finally, when $g(x)$ is irreducible, we write $g(x) = x^3 + ax^2 + bx + c$. We denote the cardinality of the field as $q$ and the characteristic as $p$. 

Each obstruction is given a label of the form 
\texttt{c.n.r.i}, where \texttt{c} is the characteristic modulo 2, \texttt{n} is the number of factors of $f(x)$ over $\Z$, \texttt{r} is the $p$-rank, and \texttt{i} is used to index obstructions within a given type. If the obstruction has no requirement for a descriptor, or the requirement is not trivially stated, we write ``N''. We also provide the number of Weil polynomials of abelian varieties over finite fields up to $\F_{25}$ for which each rule applies. Note that we use $\equiv_n$ to mean equivalence modulo $n$. Also note that \textbf{NP} refers to the Newton Polygon.

\setlength\LTleft{-0.5in} \label{table}
\setlength\LTright{-0.5in}
\setlength\LTcapwidth{\linewidth}
\renewcommand{\arraystretch}{1.2}
\begin{longtable}{@{\extracolsep{\fill}} C{0.7in} C{0.5in} C{0.6in} C{0.72in} C{0.9in} C{2in} C{0.6in} C{0.4in}} 

  \toprule
    \textbf{Label} & \textbf{Parity} & \textbf{Factors}  & \textbf{$p$-rank} & \textbf{Other} & \textbf{Rule} & \textbf{Isog Cl} & \textbf{Proof}\\
    \midrule\midrule
    
    \endfirsthead

    \multicolumn{7}{c}%
      {{\bfseries Obstructions (continued)}} \\
    \toprule
    \textbf{Label} & \textbf{Parity} & \textbf{Factors}
      & \textbf{$p$-rank} & \textbf{Other}
      & \textbf{Rule} & \textbf{Isog Cl} & \textbf{Proof}\\
    \midrule
    \endhead
    \textbf{0.N.N.0} &
  $E$ & $\{1,2,3\}$ & $\{0,1,2,3\}$
    & 
    & $(s,t,u)\equiv_2(0,1,1),(1,0,1)$ or $(1,1,0)$
    & 33{,}370 & partial
  \\ \hline
    \textbf{0.N.0.0} &
  $E$ & $\{1,2,3\}$ & $0$ 
    &
    & NP slopes are $1/2$. See~\ref{subsec:supersingular}.
    & 164 & \checkmark
  \\ \hline
  
    \textbf{0.2.2.0} &
  $E$ & $2$ & $\{0,1,2,3\}$
    &
    & $\alpha=0,\;\epsilon=\pm3$
    & 42
  \\ \hline
\textbf{0.3.1.0} &
  $E$ & $3$ & $1$
    & $q$ nonsquare
    & $\gamma-\alpha \in\{1, \sqrt{pq}\}$ or $(\beta-\alpha,\gamma-\beta) \in \{(\sqrt{pq},1), (1,\sqrt{pq})\}$
    & 32
  \\ \hline
    \textbf{0.3.2.0} &
  $E$ & $3$ & $2$
    &
    & $\alpha+5>\gamma$, $(\alpha\le \beta\le \gamma)$
    & 90
  \\ \hline
    \textbf{0.3.2.1} &
  $E$ & $3$ & $2$
    &
    & $-p\cdot v_p(q)=\alpha$, $\beta<\gamma<p\,v_p(q)-1$
    & 24
  \\ \hline
    \textbf{0.3.2.2} &
  $E$ & $3$ & $2$
    &
    & $1-p \cdot v_p(q)<\alpha<\beta\le \gamma=p\,v_p(q)$
    & 24
  \\ \hline
    \textbf{1.N.N.0} &
  $O$ & $\{1,2,3\}$ & $\{0,1,2,3\}$
    &
    & $t\equiv_20 \text{ and } u\equiv_21$
    & 245{,}548 & \checkmark
  \\ \hline
    \textbf{1.1.0.0} &
  $O$ & $1$ & $0$
    & $p=q\equiv_31$, NP has first slope $1/3$
    & $b>-q$ and $b\text{ an even multiple of }q$
    & 6
  \\ \hline
    \textbf{1.1.0.1} &
  $O$ & $1$ & $0$
    & $q=p^2$
    & $(b=-3q\text{ and }c=\text{odd}\cdot q)$ or $(b\in\{-2q, -q\}\text{ and }|a|=2p)$
    & 34
  \\ \hline
    \textbf{1.2.1.0} &
  $O$ & $2$ & $1$
    & $q\text{ prime}$
    & $\alpha \text{ odd} \text{ and } (\delta,\epsilon)\equiv_4(0,2-2q)$
    & 52
  \\ \hline
    \textbf{1.2.2.0} &
  $O$ & $2$ & $2$
    & $q\text{ prime}$
    & $\delta\text{ odd} \text{ and } \epsilon \in\{\pm2,\pm3\}$
    & 136
 \\ \hline
    \textbf{1.3.1.0} &
  $O$ & $3$ & $1$
    & $q\text{ square}$
    & $s\text{ odd} \text{ and }|s|\le\sqrt{q}$
    & 38
  \\ \hline
    \textbf{1.3.1.1} &
  $O$ & $3$ & $1$
    & $q$ prime and $q \equiv_4 3$ and $q>3$
    & $|\alpha|\le3 \text{ and } |\gamma|\le3$
    & 24
  \\ \hline
   
    \textbf{1.3.2.0} &
  $O$ & $3$ & $2$
    & $p \neq q$
    & $\gamma = p\cdot v_p(q)$ and $(\beta-\alpha \in (1,3)$ or $(\alpha = -p\cdot v_p(q)$ and $\gamma - \beta \in (1, 3)))$
    & 64
  \\ \hline
    \textbf{1.3.2.1} &
  $O$ & $3$ & $2$
    & $q\equiv_41$
    & $(\alpha,\beta,\gamma)\in\{(-2,-2,0),(0,2,2)\}$
    & 8
  \\ \hline
    \textbf{1.3.N.0} &
  $O$ & $3$ & $\{0, 1, 2, 3\}$
    &
    & $\alpha^2+\beta^2+\gamma^2=9$
    & 78
    
  \\ \hline
  \textbf{1.3.N.1} &
  $O$ & $3$ & $\{0,1,2,3\}$
    &
    & $(\alpha,\beta,\gamma)\in\{(-4,-1,0),(0,1,4),$ $(-4,-3,0), (0,3,4),$ $(-3,-2,0),(0,2,3),$ $(-1,0,2),(-2,0,1)\}$
    & 68
  \\ \hline
    \textbf{N.N.N.0} &
  $O,E$ & $\{1,2,3\}$ & $\{0,1,2,3\}$
    &
    & Illegal point counts. Details in Subsection~\ref{subsec:point_count}.
    & 1422 & \checkmark
  \\ \hline
    \textbf{N.N.N.1} &
  $O,E$ & $\{1,2,3\}$ & $\{0,1,2,3\}$
    &
    & ``Resultant-1'' method. Details in Subsection~\ref{subsec:resultant_1}.
    & 4187 & \checkmark
  \\ \hline
    \textbf{N.N.N.2} &
  $O,E$ & $\{1,2,3\}$ & $\{0,1,2,3\}$
    & $q$ square
    & ``Type'' obstruction. Details in Subsection~\ref{subsec:type}.
    & 2328 &\checkmark
  \\ \hline
    \textbf{N.3.N.0} &
  $O,E$ & $3$ & $\{0,1,2,3\}$
    &
    & ``Discriminant'' obstruction. Details in Subsection~\ref{subsec:disc}.
    & 78 & partial
  \\ \hline
    \textbf{N.3.0.0} &
  $O,E$ & $3$ & $0$
    & $p\in\{2,3,5\}$
    & $|\alpha|\neq p\cdot v_p(q) \text{ and } |\gamma|\neq p\cdot v_p(q)$
    & 47 
  \\ \hline
    \textbf{N.3.0.1} &
  $O,E$ & $3$ & $0$
    &
    & $(\alpha,\beta)=(-p\cdot v_p(q),-p \cdot v_p(q))$ or $(\beta,\gamma)=(p\cdot v_p(q),p\cdot v_p(q))$
    & 50
  \\ \hline

\caption{A summary of all known and conjectured obstructions on the existence of a hyperelliptic curve in the isogeny class of abelian varieties defined by a given Weil polynomial.}\label{table1}
\end{longtable}

\subsection{Discussion of Results} 

From the statistics in the above table, it is clear that rules \textbf{0.N.N.0} and \textbf{1.N.N.0} are by far the most widely applicable. In \cite[Thm 2.8]{costa2020}, \citeauthor{costa2020} prove rule \textbf{1.N.N.0}: for an isogeny class of 3-dimensional abelian varieties over an odd characteristic finite field, if the Weil polynomial $x^6 + sx^5 + tx^4 + ux^3 + tqx^2 + sq^2x + q^3$ has $t \equiv 0 \pmod{2}$ and $u \equiv 1 \pmod{2}$, the isogeny class cannot contain the Jacobian of a hyperelliptic curve. We present the analogous rule \textbf{0.N.N.0} for characteristic 2 finite fields, namely that the modular conditions $(s,t,u)\equiv(0,1,1),(1,0,1)\text{ or }(1,1,0)\pmod2$ provide obstructions. We prove the result for the first two residues in \cite{borodinmayeven}. 

In general, these rules classify almost 95\% of isogeny classes that do not contain hyperelliptic Jacobians in the data currently on the \cite{lmfdb} (finite fields up to $\F_{25}$). However, the remaining isogeny classes have proved increasingly hard to distinguish using elementary conditions, so it seems exceedingly unlikely that a table of arithmetic obstructions as in \cite{Howe2009} can be achieved for genus 3. 

In the remainder of this section, we discuss a few obstructions that are special cases of known results, and survey all known results by other authors for genus 3 curves.

\subsubsection{\textbf{Point Counting}}\label{subsec:point_count}

Recall that given a genus 3 hyperelliptic curve $C$ over $\F_q$, whose Jacobian has Weil polynomial with roots $\{\alpha_i\}$, the number of points on $C$ over $\F_{q^n}$ is given by $\#C(\F_{q^n})=1+q^n-\sum \alpha_i^n$. Given a Weil polynomial $f(t)$, if for any $n$ the ``predicted" number of $\F_{q^n}$ points is negative, it cannot arise from a hyperelliptic Jacobian. Note that $\sum\alpha_i$ is the trace of Frobenius, which is the coefficient $s$ in $f(t)$, so the condition $s > 1 + q$ is an obstruction. For any integer $n$, $\sum\alpha_i^n > 1 + q^n$ is also an obstruction, which is exactly rule \textbf{N.N.N.0}. These criteria can be computed explicitly using Newton's identities, which we provide in Appendix~\ref{sec:Newton}. We show that point-counting obstructions need only be checked over low-degree extensions.

\begin{proposition}
    Given a genus 3 Weil polynomial $f$ over $\F_q$ with roots $\{\alpha_i\}$,  $\sum \alpha_i^n > 1+q^n$ is only possible for $n=1$.
\end{proposition}

\begin{proof}
    Let the roots of $f$ be $\{\alpha_k = \sqrt{q} e^{i\theta_k}\}$. Suppose $\sum \alpha_k^n > 1+q^n$ for some $n$. Then $q^{n/2}\sum e^{i n \theta_k} > 1 + q^n = q^{n/2}(q^{-n/2}+q^{n/2})$. Equivalently, $\sum e^{i n \theta_k} > q^{-n/2}+q^{n/2}$. But the real part of the left-hand-side is at most $2g$, so for genus 3, if $n>1$, the inequality is impossible for all $q\geq7$. By direct computation using the data from the LMFDB, we have verified that for fields with $q<7$,  $\sum \alpha_i^n > 1+q^n$ is possible only for $n=1$.
\end{proof}

Additionally, given $\F_{q^n} \subsetneq \F_{q^{mn}}$, the point-count identity $\#C(\F_{q^{mn}}) \geq \#C(\F_{q^n})$ must hold. This yields the additional obstructions
$\sum\alpha_i^{mn} - \sum\alpha_i^n > q^{mn}-q^n$, which can be rewritten as elementary criteria on $u,s,t,$ and $q$, using Newton's identities (Appendix~\ref{sec:Newton}).

\begin{proposition}
    Point counting obstructions of the form $\#C(\F_{q^{mn}}) < \#C(\F_{q^n})$ can occur only for $q^n \leq 8$. In particular, if $q > 2$ then this is only possible for $n = 1$. 
\end{proposition}

\begin{proof}
    The relation $q^{mn} - q^n < \sum\alpha_i^{mn} - \sum\alpha_i^n$ yields $q^{mn/2} - q^{n(1-m/2)} < \sum e^{in\theta_k} - \sum q^{\frac{n}{2} - \frac{mn}{2}}e^{i\theta_k}$. Taking real parts, we obtain $q^{\frac{mn}{2}}-q^{n(1-\frac{m}{2})} < 6(1 +q^{n(\frac{1-m}{2})})$. This is only possible for $q^n \leq 8$.
\end{proof}

Since a hyperelliptic curve is a double cover of the projective line, we obtain the natural bound $\#C(\F_{q^n}) \leq 2(q^n+1)$. Additionally, since $|\alpha_i|\leq\sqrt{q}$, we also have the bound $\#C(\F_{q^n}) \leq q^n+1+2g \sqrt{q^n},$ which is a superior bound for $q>32$ in genus 3, and is trivially always satisfied by a Weil polynomial. Thus, the obstructions $\sum\alpha_i^n < -(q^n+1)$ apply only for $q\leq32$.

\begin{proposition}
    For $q\geq7$, obstructions of the form $\sum\alpha_i^n < q^n+1$ only occur when $n=1$ for fields up to $\F_{32}$, or when $n=2$ for fields up to $\F_5$.
\end{proposition}

\begin{proof}
    Taking real parts of the above inequality, we have $-6 \leq -(q^{\frac{n}{2}} + q^{\frac{-n}{2}})$. The result follows.
\end{proof}

\subsubsection{\textbf{The Resultant-1 Method}}\label{subsec:resultant_1}

As a direct corollary of the Honda-Tate Theorem, if the Weil polynomial $f$ of an abelian variety $A$ factors nontrivially as $f = f_1f_2$, then $f_1$ and $f_2$ are the Weil polynomials of abelian varieties $B_1$ and $B_2$ such that $A$ is isogenous to $B_1 \times B_2$. There is a powerful obstruction that prevents one from ``gluing'' a pair of abelian varieties to produce a curve, called the Resultant-1 method \cite[Thm. 3.6]{howe2022deducinginformationcurvesfinite}, which is as follows. Define the resultant of two polynomials to be the determinant of their Sylvester matrix. Suppose a Weil polynomial $f(x)$ has real Weil polynomial $h(x)$ where $h(x) = h_1(x)h_2(x)$ over $\Z$ such that the resultant of $h_1$ and $h_2$ is $\pm1$. Then $f$ is not the Weil polynomial of a Jacobian.
Applying this result to 3-dimensional abelian varieties, we have the following two cases:
\begin{enumerate}
    \item If $f$ can be factored as the product of an irreducible quadratic and an irreducible quartic over $\Z$, then the associated real Weil polynomial is $h_1(x)h_2(x) = (x + \alpha)(x^2 + \delta x + \epsilon)$. Computing the resultant of $h_1$ and $h_2$ gives $\alpha^2 - \alpha\delta + \epsilon = \pm1$ as an obstruction.
    \item If $f$ can be factored as $f(x) = (x^2 + \alpha x + q)(x^2 + \beta x + q)(x^2 + \gamma x + q)$ over $\Z$, then the associated real Weil polynomial is $h(x) = (x + \alpha)(x + \beta)(x + \gamma)$. Computing the resultants of the three decompositions of $h$ as a product of two polynomials, any of $(\beta - \alpha)(\gamma - \alpha) = \pm 1$, $(\alpha - \beta)(\gamma - \beta) = \pm 1$, or $(\alpha - \gamma)(\beta - \gamma) = \pm 1$ are obstructions.
\end{enumerate}

It turns out some of the other rules can actually be explained by the resultant-1 method as well. We list these partial explanations below:

\begin{itemize}
    \item In rule \textbf{0.3.1.0}, in the first case ($\gamma - \alpha = 1$), we have either $\alpha = \beta$ and $\gamma = \alpha + 1$ or $\beta = \alpha + 1$ and $\gamma = \beta$. This gives real Weil polynomial factorings of $(x+\alpha)^2(x + \alpha + 1)$ and $(x + \alpha)(x + \alpha + 1)^2$, both of which give a Resultant-1 obstruction.
    \item In rule \textbf{0.3.1.1}, because the $p$-rank is 1, two of $\alpha, \beta, \gamma$ must be 0. If the nonzero value is $\pm 1$, this rule is explained by the Resultant-1 method. 
\end{itemize}

\subsubsection{\textbf{Type Obstructions}}\label{subsec:type}
 In \cite{howe2007improvedupperboundsnumber}, Howe and Lauter define the ``type" $[x_1, x_2, x_3]$ of an abelian variety, where $x_i = -(\alpha_i + \bar{\alpha_i})$ for the pairs $\alpha_i, \bar{\alpha_i}$ of roots of the corresponding Weil polynomial. Corollary 12 of the same paper provides the following obstruction. There is no genus 3 curve over $\F_{q^2}$ of type $[2q,2q,2q-n]$, where $n$ is squarefree and prime to $q$. In \cite[Thm 3.15]{howe2022deducinginformationcurvesfinite}, Howe provides the following generalization of this result. Over a field $\F_{q^2}$, if a real Weil polynomial can be written as $h_0(x)(x\pm2q)^n$ for some $n$, where $h_0$ is a nonconstant ordinary real Weil polynomial, and $h_0(2q)$ is squarefree, then every principally polarized variety in the isogeny class is decomposable. In \S 3.2 of the same paper, it is noted that every principally polarized Jacobian variety is indecomposable; thus the above hypothesis serves as the desired obstruction.

Note that this obstruction can also be used to partially explain some of the other obstructions. For example, in rule \textbf{0.3.2.0}, if $q = s^2$ for an integer $s$, and $\alpha = -2s$, $\beta = \alpha + 1$, $\gamma = \alpha + 3$, this obstruction applies. However, this does not seem to be a very significant overlap. 

\subsubsection{\textbf{Discriminant Obstructions}}\label{subsec:disc}

Consider the case in which a degree $2g$ $q$-Weil polynomial factors as $f_E^g$, where $f_E$ is the Weil polynomial of an elliptic curve $E/\F_q$ . In this section we discuss known obstructions on which curves $E$ yield a Weil polynomial of the form $f_E^3$. Given a Weil polynomial over $\F_q$ of the form $(x^2+\alpha x+q)^3$, where we note that by the Honda-Tate theorem $x^2+\alpha x+q$ necessarily describes the Weil polynomial of an elliptic curve $E$, the usual discriminant of $E$ is defined as $d =\alpha^2-4q$. Serre \cite{SERRE1985} showed that for genus 3, $d = -3 \text{ or } -4$ is an obstruction, using a result of Beauville. Further, $d=-8  \text{ or } d=-11$ were shown to be obstructions by \citeauthor{ZAYTSEV2014} in \cite{ZAYTSEV2014, ZAYTSEV2016}. We conjecture the following longer list of values of such obstructions: $d \in \{0,-3, -4, -7, -8, -11,$ $-19, -20,-23,-27,-35,$ $-39,-43, -51,-59,-67,$ $-75,-83,-91,-99\}$.

\subsubsection{\textbf{Characteristic 2 Supersingular Curves}}\label{subsec:supersingular}

Hyperelliptic curves over finite fields of characteristic 2 take the form $y^2 + h(x)y = f(x)$, where $\deg(h) \leq g$ and $\deg(f) = 2g+1 \text{ or } 2g+2$. 

\begin{proposition} \label{prop:sschar2}
    Given a Weil polynomial $f(x) = x^6+sx^5+tx^4+ux^3+qtx^2+q^2sx+q^3$ of 2-rank 0, which arises from the Jacobian of a hyperelliptic curve $C$ of the above form over the field $\F_q$ where $q=2^r$, the Newton polygon has first slope $\frac13$.
\end{proposition}

The above result follows from \cite[Thm. 1.2]{Scholten2001}. Thus, if a Weil polynomial has a Newton polygon with all slopes equal to $\frac12$, we may conclude that it does not arise from an isogeny class containing a hyperelliptic Jacobian. The converse of Proposition ~\ref{prop:sschar2} appears to be true from inspection of the data. 
\begin{conjecture}
    Every 3-dimensional abelian variety over a finite field of characteristic 2 whose Weil polynomial has Newton polygon with first slope $\frac{1}{3}$ (equivalently satisfies $v_p(u) \geq r$) contains a genus 3 hyperelliptic Jacobian.
\end{conjecture}

\begin{remark}\label{rmk:supersingular}
    It appears that over finite fields of characteristic 2, isogeny classes of supersingular curves often do contain Jacobians of curves \cite{nart2006}, even though the above restricts these from being hyperelliptic. In particular, by \cite[Thm. 4.2]{nart2006}, all isogeny classes of simple supersingular abelian threefolds contain Jacobians, in contrast to Proposition ~\ref{prop:sschar2} for hyperelliptics. 
\end{remark}

\begin{remark}\label{rmk:finalremarks} 
    We observed, rather counterintuitively, that whether the factors of an abelian variety are isogenous to Jacobians is unrelated to whether the product variety is isogenous to a Jacobian. In particular, there exist 3-dimensional isogeny classes that contain a Jacobian which factor as the product of an isogeny class containing a Jacobian and an isogeny class not containing a Jacobian  \cite[\href{https://www.lmfdb.org/Variety/Abelian/Fq/3/23/c_e_do}{3.23.c\_e\_do}]{lmfdb})).  There also exist 3-dimensional isogeny classes not containing a Jacobian which factor into an isogeny class containing a Jacobian and not containing a Jacobian e.g. \cite[\href{https://www.lmfdb.org/Variety/Abelian/Fq/3/25/f_ay_ajl}{3.25.f\_ay\_ajl}]{lmfdb}.
\end{remark}

\section{Statistics and Asymptotics} \label{sec:StatAndAsym}

In this section we discuss the applicability of our rules according to the data for isogeny classes of abelian varieties over finite fields up to $\F_{25}$, and the asymptotic prevalence of the rules as $q \rightarrow \infty$. For complete statistics see Appendix~\ref{sec:stats}.

We conjecture that as $q \rightarrow \infty,$ for $q$ odd prime powers (resp. even prime powers), the rules \textbf{1.N.N.0} (resp. \textbf{0.N.N.0}) classify asymptotically 100\% of non-hyperelliptic Weil polynomials. To this end, we make the following claims.

\begin{conjecture} \label{conj:asymp}
    Let $\Pi_3(q)$ denote the proportion of isogeny classes of 3-dimensional abelian varieties over $\F_q$ which do not contain the Jacobian of a (genus 3) hyperelliptic curve. Let $q$ be an odd (resp. even) prime power, and let $\chi_3(q)$ denote the proportion of degree 6 $q$-Weil polynomials over $\F_q$ for which rule \textbf{1.N.N.0} (resp. \textbf{0.N.N.0}) is true. We claim the following:
    \begin{align}
    \lim_{\substack{q\to\infty \\ q\ \mathrm{odd}}} \Pi_3(q) 
    &= \lim_{\substack{q\to\infty \\ q\ \mathrm{odd}}} \chi_3(q) 
    = \frac{1}{4}, \label{eq:odd}\\
    \lim_{\substack{q\to\infty \\ q\ \mathrm{even}}} \Pi_3(q) 
    &= \lim_{\substack{q\to\infty \\ q\ \mathrm{even}}} \chi_3(q) 
    = \frac{1}{2}. \label{eq:even}
    \end{align}
    
\end{conjecture}

\begin{corollary}
    The proportion of isogeny classes of abelian varieties over $\F_q$ which are correctly classified by applying rules \textbf{1.N.N.0} and \textbf{0.N.N.0} to the associated Weil polynomials is asymptotically $100\%$ as $q\to\infty.$
\end{corollary}

In equation (1) of Conjecture~\ref{conj:asymp}, the second equality is the same as Cor. 3.2 in \cite{costa2020}, and the same paper conjectures the first equality as well. We prove the second equality of equation (2) in \cite{borodinmayeven}, and the first equality is left as conjectured. We note that these two statements together imply that the remaining rules vanish in probability as $q \to \infty$, which is evident upon inspection of their limiting behaviors; however, they remain useful for small $q$. 

Independently of this statement, we also consider the distribution of isogeny classes based on $p$-rank and factoring type. 

Concerning $p$-rank, we have the following corollary of \cite[Thm. 1.2]{dipippo2000realpolynomialsrootsunit}:

\begin{corollary}\label{lemma:rank}
    As $q \to \infty$, the ratio of the number of non-ordinary isogeny classes of $g$-dimensional abelian varieties to the total number of isogeny classes of $g$-dimensional abelian varieties is 0.
\end{corollary} 

Now, we count the number of isogeny classes that factor in a certain way. Using the asymptotic approximation found in~\cite[Thm. 1.1]{dipippo2000realpolynomialsrootsunit}, we can compute the expected number of isogeny classes of a given factoring type. 

\begin{definition}
    Let us define
    \[\mathcal{I}(g, q) = \left(\frac{2^g}{g!} \prod_{i=1}^{g} \left(\frac{2i}{2i - 1}\right)^{g + 1 - i}\right) r(x)q^{g(g+1)/4},\]
    where $r(x) = \phi(x)/x$ with $\phi(x)$ being Euler's totient function, as in~\cite[Thm.~1.1]{dipippo2000realpolynomialsrootsunit}. Further, denote $\#(2,2,2)(q)$, $\#(2, 4)(q), \#(6)(q)$ to be the predicted number of 3-dimensional isogeny classes which split as three elliptic curves, an elliptic curve and an abelian surface and a simple abelian threefold for a given field $\F_q$, respectively.
\end{definition}

\begin{lemma} \label{lemma:counting}
    We can express these quantities as:
    \begin{equation*}
    \begin{split}
        \#(2,2,2)(q) &= \binom{\mathcal{I}(1, q) + 2}{3} \sim q^{3/2}\\
        \#(2, 4)(q) &= \mathcal{I}(1, q) \left(\mathcal{I}(2, q) - \binom{\mathcal{I}(1, q) + 1}{2} \right) \sim q^2\\
        \#(6)(q) &= \mathcal{I}(3, q) - \#(2, 4)(q) - \#(2,2,2)(q) \sim q^3,
    \end{split}
    \end{equation*}
\end{lemma}
\begin{proof}
    Note that the set of abelian threefolds that split into three elliptic curves is in bijection with unordered triplets of elliptic curves, and the set of abelian threefolds that split into an elliptic and an abelian surface is in bijection with the set of pairs of an elliptic curve and a simple abelian surface. Then, the set of simple abelian threefolds is simply all of the remaining threefolds. Thus, using the asymptotic bounds on $N(g, q)$ found in \cite[Thm. 1.1]{dipippo2000realpolynomialsrootsunit} and $k$-multi-combination formulas, the result follows. 
\end{proof}

\begin{corollary} \label{cor:asymp}
    As a result of Lemmas~\ref{lemma:counting} and~\ref{lemma:rank}, as $q \rightarrow \infty$, if one is only interested in the asymptotic behavior, the set of isogeny classes of simple ordinary abelian varieties dominates. That is, the proportion of isogeny classes that are simple and ordinary approaches 1 as $q \to \infty$. 
\end{corollary}

Given the challenges in curating a complete classification of all isogeny classes of abelian threefolds, by Corollary~\ref{cor:asymp}, it is clear that in the limit $q \to\infty$ it may be more sensible to attempt a classification of only the simple ordinary abelian varieties. In the limiting case, we can disregard all rules in Table~\ref{table} that do not apply to simple Weil polynomials of $p$-rank 3. There are exactly five such rules, namely \textbf{0.N.N.0}, \textbf{1.N.N.0}, and the three additional classes of rules discussed in Sections~\ref{subsec:point_count},~\ref{subsec:resultant_1}, and~\ref{subsec:type}.

\newpage
\appendix

\section{Newton's Identities} \label{sec:Newton}

We include the following identities of symmetric polynomials due to Newton, which can be used to write the point-counting obstructions discussed in Section~\ref{subsec:point_count} in terms of the Weil parameters. Let $p_k(x_1,\dots,x_n) = \sum_{i=1}^{n}x_i^k$ denote the $k$-th power sum in $n$ variables. Let $s_k(x_1,\dots,x_n) = \sum_{i_1<\dots<i_k} (x_{i_1}\dots x_{i_k})$  denote the $k$-th elementary symmetric polynomial in $n$ variables. Given a Weil polynomial $p(x) = x^6+sx^5+tx^4+ux^3+qtx^2+q^2sx+q^3$ over $\F_q$ with roots $\alpha_1,\ldots,\alpha_6$, the coefficients satisfy the usual identities: $s=s_1(\alpha_1,\ldots,\alpha_6),$ $t = s_2(\alpha_1,\ldots,\alpha_6),$ $u = s_3(\alpha_1,\ldots,\alpha_6),$ etc. We provide a list of the Newton identities in terms of coefficients $s,t,u$ and $q$.

\begin{enumerate}
    \item $p_1 = \sum \alpha_i = s.$
    \item $p_2 = \sum \alpha_i^2 = s^2-2t.$
    \item $p_3 = \sum \alpha_i^3 = s^3-3st+3u.$
    \item $p_4 = \sum \alpha_i^4 = s^4 - 4s^2t+2t^2+4su-4qt.$
    \item $p_5 = \sum \alpha_i^5 = s^5-5s^3t+5s^2u+5st^2-5tu-5qst+5q^2s.$
    \item $p_6 = \sum \alpha_i^6 = s^6-6s^4t+6s^3u+9s^2t^2-6qs^2t-12stu+6q^2s^2-2t^3+3u^2+6qt^2-6q^3$
\end{enumerate}

\newpage
\section{Statistics} \label{sec:stats}

We provide statistics on the obstructions enumerated in Section~\ref{sec:summary}. In Table~\ref{tab:rule_occurrences}, we list for each obstruction the total number of Weil polynomials over the fields of size $\leq25$ that satisfy the rule, as well as the number of Weil polynomials that satisfy only this rule and no others. Notably, none of the rules are superfluous. In Table~\ref{tab:category_detection}, for each category of Weil polynomials we give the proportion of isogeny classes of abelian threefolds over fields of size $\leq 25$ that do not contain hyperelliptic Jacobians that are not identified by any of the obstructions in Section~\ref{sec:summary}. The type is \texttt{(number of irreducible factors, $p$-rank)}.

\begin{table}[!htb]
\centering
\begingroup
\renewcommand{\arraystretch}{1.2} 
\setlength{\tabcolsep}{6pt}       

\begin{tabular}{@{}c@{\hspace{0.2em}}c@{}} 
\begin{minipage}[t]{0.48\textwidth}
\centering
\begin{tabular}[t]{|c|c|c|}
  \hline
  \textbf{Rule} & \textbf{hits} & \textbf{unique hits} \\ \hline 
  \textbf{0.2.2.0} & 42 & 14 \\ \hline
  \textbf{0.3.1.0} & 32 & 4 \\ \hline
  \textbf{0.3.2.0} & 90 & 32 \\ \hline
  \textbf{0.3.2.1} & 24 & 6 \\ \hline
  \textbf{0.3.2.2} & 24 & 6 \\ \hline
  \textbf{0.N.0.0} & 164 & 86 \\ \hline
  \textbf{0.N.N.0} & 33370 & 32855 \\ \hline
  \textbf{1.1.0.0} & 6 & 6 \\ \hline
  \textbf{1.1.0.1} & 34 & 4 \\ \hline
  \textbf{1.2.1.0} & 52 & 49 \\ \hline
  \textbf{1.2.2.0} & 136 & 132 \\ \hline
  \textbf{1.3.1.0} & 38 & 24 \\ \hline
  \textbf{1.3.1.1} & 24 & 8 \\ \hline
  \textbf{1.3.2.0} & 64 & 42 \\ \hline
  \textbf{1.3.2.1} & 8 & 8 \\ \hline
  \textbf{1.3.N.0} & 78 & 48 \\ \hline
  \textbf{1.3.N.1} & 68 & 62 \\ \hline
  \textbf{1.N.N.0} & 245548 & 245250 \\ \hline
  \textbf{N.3.0.0} & 47 & 4 \\ \hline
  \textbf{N.3.0.1} & 50 & 10 \\ \hline
  \textbf{N.3.N.0} & 113 & 4 \\ \hline
  \textbf{N.N.N.0} & 1422 & 509 \\ \hline
  \textbf{N.N.N.1} & 4187 & 3543 \\ \hline
  \textbf{N.N.N.2} & 2328 & 1819 \\ \hline
\end{tabular}

\vspace{0.8em} 
\caption{Rule occurrences and unique hits.}
\label{tab:rule_occurrences}
\end{minipage}
&
\begin{minipage}[t]{0.48\textwidth}
\centering
\begin{tabular}[t]{|c|c|} 
  \hline
  \textbf{type} & \textbf{Undetected by rules} \\ \hline 
  (3, 0) & 8/142 = 0.0563 \\ \hline
  (3, 1) & 183/416 = 0.4399 \\ \hline
  (3, 2) & 108/659 = 0.1639 \\ \hline
  (3, 3) & 225/1524 = 0.1476 \\ \hline
  (2, 0) & 24/103 = 0.2330 \\ \hline
  (2, 1) & 576/870 = 0.6621 \\ \hline
  (2, 2) & 1264/6117 = 0.2066 \\ \hline
  (2, 3) & 3808/15224 = 0.2501 \\ \hline
  (1, 0) & 0/122 = 0.0000 \\ \hline
  (1, 1) & 154/1040 = 0.1481 \\ \hline
  (1, 2) & 2758/27174 = 0.1015 \\ \hline
  (1, 3) & 6009/247844 = 0.0242 \\ \hline
  \textbf{Total} & 15117/301235 = 0.0502 \\ \hline
\end{tabular}

\vspace{0.8em} 
\caption{Category detection \\rates by type.}
\label{tab:category_detection}
\end{minipage}
\end{tabular}

\endgroup
\end{table}

\newpage

\pagebreak

\printbibliography

\end{document}